\documentclass[12pt]{amsart}
\usepackage{amsmath,amsfonts,amsthm,amscd,stmaryrd, amssymb,mathrsfs}
\usepackage[all,cmtip]{xy}
\usepackage{enumitem}
\usepackage{hyperref}
\usepackage{tikz, tikz-cd}
\newtheorem{theorem}{Theorem}[section]

\newtheorem{proposition}[theorem]{Proposition}
\newtheorem{lemma}[theorem]{Lemma}
\newtheorem{corollary}[theorem]{Corollary}
\theoremstyle{definition}
\newtheorem{definition}[theorem]{Definition}
\newtheorem{remark}[theorem]{Remark}
\newtheorem{example}[theorem]{Example}

\newcommand{\aaa}{\alpha}

\newcommand{\CCC}{\Gamma}

\newcommand{\DDD}{\Delta}

\newcommand{\id}{{\rm{id}}}

\newcommand{\PP}{\mathbb{P}}
\newcommand{\CC}{\mathbb{C}}

\newcommand{\RR}{\mathbb{R}}

\newcommand{\QQ}{\mathbb{Q}}

\newcommand{\rank}{{\rm{rank}}}
\newcommand{\corank}{{\rm{corank}}}
\newcommand{\Sing}{{\rm{Sing}\,}}

\newcommand{\ol}{\overline}

\newcommand{\lras}{\,\longrightarrow\,}

\newcommand{\set}{\,|\,}

\newcommand{\inv}{^{-1}}

\newcommand{\ms}{\mathscr}
\newcommand{\minus}{\backslash}
\newcommand{\ptl}{\partial}

\newcommand{\qandq}{\quad{\text{and}}\quad}

\renewcommand{\hat}{\widehat}
\renewcommand{\tilde}{\widetilde}

\usepackage[letterpaper]{geometry}
\numberwithin{equation}{section}
\begin{document}
\title[Some remarks on fibrations in complex geometry]{Some remarks on fibrations in complex geometry}
\author{Nobuhiro Honda}\address{Department of Mathematics, Tokyo Institute of Technology, Tokyo, Japan}\email{honda@math.titech.ac.jp}
%\thanks{Partially supported by JSPS KAKENHI Grant 22K03308.}
\author{Jeff Viaclovsky}\address{Department of Mathematics, University of California, Irvine, USA}\email{jviaclov@uci.edu}
%\thanks{Partially supported by NSF Grant DMS-2105478 and a Fellowship from the Simons Foundation.}
\date{April 19, 2025}
\dedicatory{To Gang Tian on the occasion of his 65th birthday}
\begin{abstract} 
In this article, we discuss some properties of holomorphic fibrations in the complex analytic setting.
\end{abstract}
\maketitle

\section{Introduction}

In this article, we will discuss some general properties of holomorphic fibrations and a give some applications. We will also prove a naturality property of the Leray spectral sequence with respect to open inclusions.

\subsection{Holomorphic fibrations}
The main objects of interest in this article are the following. 
\begin{definition}
\label{d:fibration}
Let $Z$ and $Y$ be reduced and irreducible complex spaces satisfying $1 \leq \dim(Y) < \dim(Z)$.  A \textit{fibration} $f: Z \rightarrow Y$ is a proper surjective holomorphic mapping with connected fibers. We say that $Z$ \textit{fibers} over $Y$ and call $Y$ the \textit{base} of the fibration. The \textit{fiber} of $f$ over a point $y$ is the set $f^{-1}(y)$ with the fiber product structure, and is denoted by $Z_y = Z \times_{Y} \mbox{spec} (k(y))$. 
\end{definition}

In Section~\ref{s:fibrations} we will discuss some general properties of fibrations, such as fiber dimension, discriminant locus, monodromy, ordinary points, and non-ordinary points. 
Next, in Section~\ref{s:curve}, we will discuss the case where the base $Y$ is a curve $C$. In this case, if $Z$ is smooth and compact, then the discriminant locus is a finite set, so there are only finitely many singular fibers. For any topological space $X$, we let $\chi(X)$ denote the topological Euler characteristic. The monodromy is defined in Subsection~\ref{ss:monodromy}. We will prove the following estimate on the first betti number $b^1(Z_y)$ of the singular fibers.  
\begin{theorem}
\label{t:t1} Let $f : Z \rightarrow C$ be a fibration, where $Z$ is a compact complex $n$-manifold and $C$ is a smooth curve. Then 
\begin{align}
\label{b1est}
b^1(Z_y) \leq b^1(F)
\end{align}
for \textit{all} $y \in C$, where $F$ is the general fiber. Equality holds in \eqref{b1est} for some $y \in C$ if and only if the local monodromy $\rho^{(1)}_y$ around $y$ is trivial. 

If $n = 2$, then we have the inequality
\begin{align}
\label{eulerineq}
\chi(Z) \geq \chi(F) \chi(C),
\end{align}
with equality if and only if the global monodromy $\rho^{(1)}$ is trivial and every fiber of $f$ is irreducible. 
\end{theorem}
For $n=2$, this was previously proved in \cite[Chapter~III.11]{BHPV}, but the higher-dimensional case of \eqref{b1est} seems new.  We note that if $Z$ were assumed to be K\"ahler, then the local invariant cycles theorem is true in higher dimensions, in which case Theorem~\ref{t:t1} is known; see \cite{Cl77}. In Subsection~\ref{ss:ict}, we will also give some application to a local invariant cycles theorem in the non-K\"ahler case; see Corollary~\ref{c:ict}. 

In Section~\ref{s:sing}, we will discuss some properties of normal surface singularities. In particular, we will define the notion of a \textit{rational tree singularity}, which is generalization of a rational singularity; see Definition~\ref{d:tsing}.

We next turn our attention to the case where $Z$ is a threefold and $Y$ is a surface. 
In Section~\ref{s:sing}, we will prove the following result in the case $Z$ is a conic bundle over a surface with rational tree singularities. 
\begin{theorem}
\label{t:t2}
 Let $f : Z \rightarrow Y$ be a fibration, where $Z$ is a compact connected threefold and $Y^2$ is a compact surface with rational tree singularities. Assume that the generic fiber $F$ of $f$ is a smooth rational curve.
Then $b^1(Z_y) = 0$ for all $y \in Y$. 
If in addition the fibers of $f$ are equidimensional, then
\begin{align}
\label{t2chi}
\chi(Z) \geq 2 \cdot \chi(Y),
\end{align}
with equality if and only if every fiber is homeomorphic to $\PP^1$. 
\end{theorem}
In the case of equality in \eqref{t2chi} it is tempting to speculate that every fiber would have to be a smooth $\PP^1$, 
but we are unable to make such a strong conclusion with our methods. 
\begin{remark}If $Z$ and $Y$ were assumed to be projective with $Y$ only having rational singularities, then this result can be seen to follow from \cite[Theorem~7.1]{KollarI} (where even stronger results are obtained). We emphasize that in Theorem~\ref{t:t2}, we do not make any projectivity or K\"ahlerian assumption. 
\end{remark}
The main tool in the proof of Theorem~\ref{t:t2} is a certain naturality property of the Leray spectral sequence, which we discuss next. 

\subsection{Naturality of Leray spectral sequence}
Let $M$ and $N$ be topological spaces and $f : M \rightarrow N$ a continuous mapping. 
Given a sheaf of Abelian groups $\mathscr{S}$ on $M$, there is the well-known Leray spectral sequence with
\begin{align}
E_2^{p,q} = H^p(N, R^q f_* \ms{S}) \Longrightarrow H^{p+q}(M, \ms{S}). 
\end{align}
If $N' \subset N$ is an open subset, we let $M' = f^{-1}(N')$, and denote the restrictions by $f' : M' \rightarrow N'$ and $\ms{S}' = \ms{S}|_U$. We also have a Leray spectral sequence 
\begin{align}
{E}_2'^{p,q} = H^p(N', R^q f'_* \ms{S}') \Rightarrow H^{p+q}(M', \ms{S}'). 
\end{align}
In Section \ref{s:Leray} we will prove the following naturality property of the Leray spectral sequence.
\begin{theorem}\label{t:Leray}
For any $p\ge 0$ and $q\ge 1$, there exists a commutative diagram 
\begin{equation}
\begin{CD}\label{}
H^p\big(N,R^qf_*\ms S\big) @>{d_2}>> H^{p+2}\big(N,R^{q-1}f_*\ms S\big)\\
@VVV @VVV\\
H^p\big(N',R^qf'_*\ms S'\big) @>{d'_2}>> H^{p+2}\big(N',R^{q-1}f'_*\ms S'\big)
\end{CD}
\end{equation}
where $d_2$ and $d'_2$ are differentials in the second terms in the Leray spectral sequences associated with the pairs $(\ms S,f:M\rightarrow N)$ and $(\ms S',f':M'\rightarrow N')$ respectively, and the vertical mappings are the restriction mapping from $N$ to $N'$ composed with the isomorphism $(R^if_*\ms S)'\simeq R^if'_*\ms S'$ for any $i$ as in Proposition \ref{p:002}.
\end{theorem}
While this is a known folklore result, we could not find any elementary explanation in the literature, so we provide a detailed proof. 

There is a $5$-term exact sequence associated with the Leray spectral sequence 
\begin{align}
0 \lras E_2^{1,0} \lras H^1(M,\ms S) \lras E_2^{0,1} 
\overset{d_2}{\lras} E_2^{2,0} \lras H^2(M, \ms S);
\end{align}
see \cite[Chapter~5]{Weibel}. The above result implies the following corollary.
\begin{corollary}
\label{c:5term} 
In the above setting, we have a commutative diagram 
\begin{equation}
\begin{tikzcd} 
 E_2^{1,0} \arrow[r] \arrow[d] &  H^1(M,\ms S) \arrow[r] \arrow[d] &  E_2^{0,1} 
\arrow[r,"d_2"] \arrow[d]  & E_2^{2,0} \arrow[r] \arrow[d] & H^2(M, \ms S) \arrow[d]\\
E_2'^{1,0} \arrow[r]  &  H^1(M',\ms S') \arrow[r]  &  E_2'^{0,1} 
\arrow[r,"d_2'"]  & E_2'^{2,0} \arrow[r]  & H^2(M', \ms S'),
\end{tikzcd}
\end{equation}
where the vertical maps are those induced by restriction. 
\end{corollary}
Some applications of this commutativity will be given in Section~\ref{s:tfc}; see Theorem~\ref{t:tf1}.

\subsection{Acknowledgements}  The first author was partially supported by JSPS KAKENHI Grant 22K03308. The second author was partially supported by NSF Grants DMS-2105478, DMS-2404195, and a Fellowship from the Simons Foundation.  The second author would like to thank Gang Tian for his collaboration, friendship, and generosity throughout the years since they first met over 25 years ago.

\section{Holomorphic fibrations}
\label{s:fibrations}
Let $f: Z \rightarrow Y$ be a fibration as in Definition~\ref{d:fibration}. In this section, we will review some fundamental definitions and properties of $f$. 
\subsection{Fiber dimension}
\label{ss:fd}
We begin with the following basic result on the dimension of the fibers of $f$. As in \cite{U75}, we define  $\dim(f) = \dim(Z) - \dim(Y)$.

\begin{proposition} If $Z$ is irreducible, then $\dim_z(Z_y) \geq \dim(f)$ for all $y \in Y$ and all $z \in Z_y$. Furthermore, the set of $y \in Y$ for which there exists $z \in f^{-1}(y)$ with $\dim_z(Z_y) > \dim(f)$ is a subvariety of $Y$ of codimension at least $2$.
\end{proposition}
\begin{proof}
This is proved in \cite{Remmert1957}; see also \cite[Chapter~3]{Fischer}.
\end{proof}
If all irreducible components of $Z_y$ are of dimension $\dim(f)$ for all $y \in Y$, then we say that $f$ has \textit{equidimensional} fibers. 
We mention that if $Z$ is Cohen-Macaulay, $Y$ is regular, and $f$ has equidimensional fibers, then $f$ is flat. This is known as ``miracle flatness''; see  \cite[Exercise III.10.9]{Hartshorne} or \cite[Theorem~23.1]{Matsumura}. However, we will not directly use any flatness conditions in the present article.

\subsection{Discriminant locus}
\label{ss:dl}
Next, we have the following result which defines the discriminant locus of a fibration. 
\begin{proposition} 
\label{p:disc}
Let $f : Z \rightarrow Y$ be a fibration such that $f(\Sing Z) \neq Y$. 
Then there exists a analytic subset $\mathscr{D} \subset Y$ of codimension at least $1$ in $Y$ with $\Sing Y \subset \mathscr{D}$
and $\Sing Z \subset f^{-1}( \ms{D})$, called the {\textit{discriminant locus}} of $f$, such that the restriction 
\begin{align}\label{fb1}
f' : Z\setminus f^{-1}(\ms{D})  \rightarrow Y\setminus  \ms{D}
\end{align}
is a smooth submersion.
\end{proposition}
\begin{proof} 
If $Y$ and $Z$ were non-singular, then the critical locus of a surjective holomorphic mapping $f:Z\rightarrow Y$ is defined as the set 
\begin{align}\label{clc}
\ms C \equiv \big\{z\in Z\set \rank\, (df)_z < \dim Y\big\},
\end{align}
and the discriminant locus of $f$ is defined as $f(\ms C)$, which is an analytic subset of $Y$ by the Remmert proper mapping theorem; see \cite[p.~213]{GR}.

Let $\ms{D}_0 = \Sing(Y)$, which is a proper subvariety of $Y$; see \cite{GR}.  
Similarly $\Sing(Z)$ is a proper subvariety of $Z$, and by the Remmert proper mapping theorem, $\ms{D}_1 = f(\Sing(Z))$ is a subvariety of $Y$. By the assumption $\ms{D}_1$ must be a proper subvariety of $Y$. So we consider the restricted mapping 
\begin{align}
f : Z \setminus f^{-1} (\ms{D}_0 \cup \ms{D}_1) 
 \rightarrow Y \setminus ( \ms{D}_0 \cup \ms{D}_1) .
\end{align}
Since we have removed all the singular points,  this is a mapping between smooth manifolds. 
So $f$ will be a submersion away from a proper subvariety $\ms{D}_2$ of 
$Y \setminus ( \ms{D}_0 \cup \ms{D}_1)$.  
Then $\ms{D} = \ms{D}_0 \cup \ms{D}_1 \cup \ms{D}_2$ is the required set, provided that it is a subvariety of $Y$. 

To show that $\ms{D}$ is a subvariety of $Y$, we argue as follows. 
Near $z \in Z$, we can locally identify a neighborhood of $z$ with a subvariety of $\CC^N$ and similiary near $f(z)$. Then $df: T_z Z \rightarrow T_{f(z)} Y$ is still defined as a mapping between the Zariski tangent spaces. Instead of considering \eqref{clc}, one instead defines the \textit{corank} of $f$ at $z$ to be  
\begin{align}
\corank_z(f) = \dim(T_z Z) - \rank(df)_z.
\end{align}
Then the critial locus is defined to be 
\begin{align}
\ms C \equiv \{ z \in Z \set \corank_z(f) > \dim(f) \}.
\end{align} 
It is shown in \cite[Theorem~L.4]{GunningII} that this is an analytic subset of $Z$, so $\ms D_3 = f(\ms C)$ is an analytic subvariety of $Y$ by the Remmert proper mapping theorem. 
Clearly, we have that $\ms D_3 \setminus (\ms D_0 \cup \ms D_1) = \ms D_2$, which implies that 
$\ms{D} = \ms{D}_0 \cup \ms{D}_1 \cup \ms{D}_3$ is a subvariety of $Y$. 
\end{proof}

\begin{remark} By Ehresmann's fibration theorem, the mapping \eqref{fb1} is a smooth fiber bundle; see \cite{Eh51}. Since $\mathscr{D}$ is an proper analytic subvariety of $Y$ and $Y$ is irreducible, 
the complement $Y\minus \mathscr{D}$ is connected by \cite[p.~145]{GR}, which implies that all fibers of the mapping \eqref{fb1} are diffeomorphic. In this case, we will refer to a fiber of the mapping in \eqref{fb1} as a \textit{general fiber}, and usually denote the general fiber by $F$.  
\end{remark}
In Section~\ref{s:tfc}, we will also require the following result. 
\begin{proposition}
\label{p:topo}
Suppose that $Z$ and $Y$ are compact and $\dim(f) = 1$. Then there exists an analytic subset $A\subsetneq\ms D$ such that $f:
f\inv(\ms D\setminus A)\lras \ms D\setminus A$ is topologically a fiber bundle mapping.
\end{proposition}
\begin{proof} 
If $\dim(Z) = 2$, then the result is trivial, so we assume that $\dim(Z) > 2$. 
Put $X=f\inv(\ms D)$. This is a subvariety of $Z$, which is reducible in general. Though $X$ is non-reduced in general, in the sequel, we consider the reduction as its analytic structure.
Let $\nu:\tilde X\to X$ be the normalization of $X$.
This in particular makes all irreducible components of $X$ disjoint.
Since all components of $\tilde X$ are normal, 
the singular locus of $X$ has codimension at least two.
The restriction of the composition $\tilde f:=f\circ\nu$ to each connected component of $\tilde X$ is also proper. By the argument in the proof of Proposition~\ref{p:disc} (which does not need connected fibers) and Ehresmann's fibration theorem, $\tilde{f}$ is differentiably a fiber bundle map away from a proper subvariety $B_1 \subset \ms D$. That is, $\tilde X\setminus\tilde f\inv(B_1)\to \ms D\setminus B_1$ is differentiably a fiber bundle whose fibers are typically disconnected.

From the fundamental property about the normalization given in \cite[p.\,163, Proposition]{GR}, $\nu:\tilde X\to X$ is naturally factorized as $\tilde X\stackrel{\tau}\lras X'\stackrel{\nu'}\lras X$, where $X'$ is the space of all prime germs of $X$, $\nu':X'\to X$ is the natural holomorphic surjection, and $\tau:\tilde X\to X'$ is the normalization of $X'$ which is a homeomorphism.
Here, one of $\tau$ and $\nu'$ can be a biholomorphic mapping. The mapping $\nu'$ is biholomorphic if and only if $X$ is irreducible at every point of $X$, and in particular, $\nu'$ already makes all irreducible components of $X$ disjoint.
Since $\tilde X\simeq X'$ topologically by $\tau$, for the mapping $f':= f\circ\nu':X'\to \ms D$,
the restriction $X'\setminus (f')\inv(B_1)\to \ms D\setminus B_1$
 is also topologically a fiber bundle.

Let $N\subset X$ be the non-normal locus of $X$, which is a subvariety of $X$.
Put $\tilde N:=\nu\inv(N)$ and $N':=(\nu')\inv(N)$. These are subvarieties of $\tilde X$ and $X'$ respectively.
The restriction $\nu|\tilde N:\tilde N\to N$ is a finite mapping, $\tau|\tilde N:\tilde N\to N'$ is a homeomorphism, and $\nu'|N':N'\to N$ is a finite map which is not of degree one unless $\nu'$ is biholomorphic.
Let $\hat N$ be the complex space which is obtained from $N$ by removing all irreducible components of $N$ which are contracted to a proper subvariety of $\ms D$ by $f$, and let $\hat N'$ be the complex space which is obtained from $N'$ by removing all inverse images of the components of $N-\hat N$.
Then the restriction of $f'$ to each connected component of $\hat N'$ is surjective over $\ms D$ and it is generically a finite mapping.
Let $B_2\subset\ms D$ be the union of the discriminant locus of the restriction of $f'$ to each connected component of $\hat N'$.
$B_2$ is also a proper subvariety of $\ms D$. 
Since $\dim(f) = 1$, the restriction $\hat N'\setminus (f')\inv(B_2)\to \ms D\setminus B_2$ of $f'$ is a finite unramified covering of complex manifolds of dimension $\dim Y-1$.

By the construction of the space $X'$, 
$X$ is obtained from $X'$ by identifying points that are over the same point of $N$ under $\nu'$. Namely, $X=X'/\nu'$.
The restriction $X'\setminus (f')\inv(B_1\cup B_2)\to \ms D\setminus (B_1\cup B_2)$ of $f'$ is topologically a fiber bundle map since it is so over $\ms D\setminus B_1$ as above.
Further, from the factorization $f' = f\circ\nu'$, the restriction $\hat N'\setminus (f')\inv(B_2)\to \hat N\setminus f\inv(B_2)$ of $\nu'$ is also an unramified covering.
Therefore, so is a further restriction $\hat N'\setminus (f')\inv(B_1\cup B_2)\to \hat N\setminus f\inv(B_1\cup B_2)$ of $\nu'$.
Since $f':X'\setminus (f')\inv(B_1)\to \ms D\setminus B_1$ was topologically a fiber bundle map, the restriction $X'\setminus (f')\inv(B_1\cup B_2)\to \ms D\setminus (B_1\cup B_2)$ is also topologically a fiber bundle map.
Combining with the last unramified property, we conclude that the restriction $f:
X\setminus f\inv(B_1\cup B_2)\to \ms D\setminus (B_1\cup B_2)$ is also topologically a bundle map.
Taking $A=B_1\cup B_2$, this finishes the proof of the proposition.
\end{proof}

\subsection{Monodromy}
\label{ss:monodromy}
A fundamental object associated with a fibration is the higher direct image sheaf 
$R^jf_* \RR$, which we next discuss in more detail. 
\begin{proposition}
\label{p:stalk}
 For all $y \in Y$, the stalk of $R^jf_* \RR$ over $y$ satisfies
\begin{align}
(R^j f_* \RR)_y \simeq H^j(Z_y;\RR).
\end{align}
\end{proposition}
\begin{proof}
The stalk of $(R^j f_* \RR)_y$ is the directed limit of $H^j(f^{-1}(V);\RR)$ as $V \to y$. Let $U$ be any open set containing $Z_y$. Recall that any proper continuous mapping to a first countable Hausdorff space is a closed mapping \cite[Corollary]{Pa70}. Since $Z \setminus U$ is closed, this implies that $f(Z \setminus U)$ is closed. Then consider the open set $V = Y \setminus f(Z \setminus U)$. Then $f^{-1}(V) \subset U$. Consequently, $(R^j f_* \RR)_y$ is the directed limit of $H^j(U;\RR)$ over \textit{all} open sets $U$ containing $Z_y$. But since $Z_y$ is a compact subset of $Z$, this limit equals $H^j(Z_y;\RR)$ since the sheaf $\RR$ is a constant sheaf; see \cite[Theorem III.6.2]{Iversen}.
\end{proof}
Assuming that $f(\Sing Z) \neq Y$, then the mapping in $\eqref{fb1}$ is a smooth fiber bundle, the sheaf $R^j f'_* \RR$ is locally constant over $Y' \equiv Y \setminus \ms{D}$ for any $j \geq 0$. Locally constant sheaves are equivalent to representations of the fundamental group acting by automorphism of any fixed fiber; see \cite[Section~IV.9]{Iversen}. Therefore, we obtain the monodromy mappings
\begin{align}
\rho^{(k)} : \pi_1( Y') \rightarrow Aut(H^k(F;\RR)).
\end{align}
In the case the base $Y$ is a curve and $y \in Y$ is an isolated discriminant point, we will denote by $\rho^{(k)}_y$ the local monodromy around $y$.  Later, we will be interested in the invariants of the monodromy representation, 
which are denoted by 
\begin{align}
H^k_{\rho}(F;\RR) = \{ v \in H^k(F;\RR) \ | \ \rho^{(k)}(\gamma)(v) = v \ \mbox{for all} \ \gamma \in \pi_1(Y')\}.
\end{align}
We note also that 
\begin{align}
\label{hkrho}
H^k_{\rho}(F;\RR) \simeq H^0(Y', R^kf_* \RR);
\end{align}
 see \cite[Proposition~5.5.14]{DavisKirk}. 
\subsection{Ordinary and non-ordinary points}
\label{ss:op}
Now let us restrict to the case that $Z$ is smooth and $Y$ is normal, so that $\Sing Y$ has codimension at least $2$ in $Y$. The following definition is adapted from \cite[p.\,29, Definition]{Fjt86}.
\begin{definition} We say that $y \in \ms{D}$ is an {\em ordinary point} if 
$y$ is regular point of $\ms D$ with $\dim_y(\ms D) = \dim(Y) -1$, and 
there exists a holomorphic non-singular disc $\Delta \subset Y$ 
through $y$, defined in a neighborhood of $x$, such that $\Delta$ intersects $\ms D$ transversally at $x$, and such that $f\inv(\Delta)$ is non-singular.
Any other point $y \in \ms D$ is called a {\em non-ordinary point} of $\ms{D}$. 
\end{definition}
Note that in particular, any singular point $y \in \Sing \ms D$ is non-ordinary
and any point $y \in \ms{D}$ with $\dim_y(\ms D) < \dim(Y) -1$ is non-ordinary.

\begin{proposition} 
\label{p:non}
The set of non-ordinary points of $\ms D$ is an analytic subset of $Y$ with codimension at least $2$. 
\end{proposition}
\begin{proof}
If $y$ is an ordinary point, then we call the curve $\Delta$ through $y$ in the definition 
an {\em ordinary slice} of $\ms{D}$ through $x$.
By Bertini's theorem, moving an ordinary slice along $\ms{D}$
in such a way that the slices are disjoint, the claim follows. 
\end{proof}

The main point is that singular fibers over ordinary points are degenerations of the general fiber over a $1$-dimensional disc. These are often easier to analyze, and will be discussed in the next section. Fibers over non-ordinary points are much more difficult to understand in general. 

\section{Fibrations over a curve}
\label{s:curve}
In this section, we consider the case where the base $Y$ is a curve. In this case, the discriminant locus consists of isolated points. We begin with a review of some basic results on deformation retractions in this setting. 

\subsection{Deformation retractions}

Given a subspace $A$ of a topological space $M$, we define a {\em deformation retraction} to be a continuous mapping $\Phi:[0,1]\times M\rightarrow M$ such that $\Phi(1,\cdot) = \id_M$ and $\Phi(0,\cdot)$ is a mapping from $M$ to $A$ that satisfies $\Phi(t,\cdot)|_A = \id_A$ for any $t\in[0,1]$.  We first recall the following result which follows easily from \cite{Cl77} with some minor remarks. 
\begin{proposition} 
\label{p:def}
Let $Z$ be a smooth complex $n$-dimensional manifold, and $\pi: Z \rightarrow \Delta$ a proper surjective holomorphic mapping with $\pi$ of maximal rank away from $F_0 = \pi^{-1}(0)$. Then there exists a deformation retraction $\Phi:[0,1]\times Z\rightarrow Z$ of $Z$ onto the fiber $F_0$ that is a lift of the standard radial deformation retraction from $\DDD$ onto $\{0\}$. Furthermore, if $F_0 = \sum_j m_j F_{0,j}$ is the decomposition of the fiber $F_0$ into irreducible components with $m_j > 0$, then 
the isomorphism $H_{2n-2}(Z)\simeq H_{2n-2}(F_0)$ induced from the deformation retraction $\Phi$ maps the fiber class of $\pi$ to the class $\sum_j m_j[F_{0,j}]$, where $[F_{0,j}]$ denotes the fundamental class of the component $F_{0,j}$.
\end{proposition}
\begin{proof} 
By applying blowups to eliminate non-transversal intersections in $F_0$, we obtain a smooth manifold $\tilde Z$ equipped with a birational holomorphic mapping $\mu:\tilde Z\rightarrow Z$, such that the composition $\tilde\pi:=\pi\circ\mu:\tilde Z\rightarrow \DDD$ satisfies the following conditions:
\begin{itemize}
\item any irreducible component of
the fiber $\tilde F_0:=\tilde \pi\inv(0)$ is non-singular,
\item any intersection of two distinct components of $\tilde F_0$ is transverse,
\item The intersection of strictly more than $n$ distinct components is empty. 
\end{itemize}
By using \cite[(5.1) and Theorem 6.9]{Cl77},
there exists a deformation retraction $\tilde\Phi:[0,1]\times\tilde Z\rightarrow \tilde Z$ of $\tilde Z$ onto the fiber $\tilde F_0$, which is a lift of the standard retraction from $\DDD$ onto a point $\{0\}$ defined by $(t,y)\mapsto ty$, where $t\in[0,1]$ and $y\in\DDD$; in particular, $\tilde\Phi$ preserves the fibration $\tilde\pi$ in the sense that for any $t\in[0,1]\times\RR$, $\tilde\Phi(t,\cdot)$ maps a fiber of $\tilde\pi$ to a fiber of $\tilde\pi$.
In fact, in \cite{Cl77}, an action of the semigroup $[0,1]\times \RR$ on $\tilde N$ that preserves the fibration $\tilde\pi$ is constructed, and the present retraction is obtained by just restricting it to $[0,1]\times\{0\}\subset [0,1]\times\RR$.

We can see that $\tilde\Phi$ descends to a retraction $\Phi$ of $Z$ onto $F_0$ as follows.
Define a deformation retraction $\Phi:[0,1]\times Z \rightarrow Z$ by
\begin{align}\label{Phi}
\Phi(t,z) =
\begin{cases}
\mu\circ\tilde\Phi(t, \mu^{-1}(z)) &  \pi(z) \neq 0\\
z & z \in \pi^{-1}(0). 
\end{cases}
\end{align}
All the properties of a deformation retraction are obvious, except 
for the continuity $\Phi$. To show this, let $A\subset Z$ be any closed subset. Then  
\begin{align}\label{Clem}
\Phi\inv(A) = (\id\times\mu)\big((\mu\circ\tilde\Phi)\inv(A)\big).
\end{align}
Since $\mu$ and $\tilde\Phi$ are continuous, $(\mu\circ\tilde\Phi)\inv(A)$ is a closed subset of $[0,1]\times \tilde Z$. Recall that any proper continuous mapping to a first countable Hausdorff space is a closed mapping; see~\cite[Corollary]{Pa70}.
Since $\mu$ is proper and continuous, the mapping $\id\times \mu$ is also proper and continuous,
so  \eqref{Clem} implies that $\Phi\inv(A)$ is closed in $[0,1]\times Z$. Therefore, $\Phi$ is continuous.

For the second property, let $\tilde{F}_0 = \sum_j \tilde{m}_j \tilde{F}_{0,j}$ be the decomposition of the fiber $\tilde{F}_0$ into irreducible components with $\tilde{m}_j > 0$. Then
the isomorphism $H_{2n-2}(\tilde{Z})\simeq H_{2n-2}(\tilde{F}_0)$ induced from the deformation retraction $\tilde{\Phi}$ maps the fiber class of $\tilde{\pi}$ to the class $\sum_j \tilde{m}_j[\tilde{F}_{0,j}]$, where $[\tilde{F}_{0,j}]$ denotes the fundamental class of the component $\tilde{F}_{0,j}$.
This is a consequence of the action of the semigroup $[0,1]\times \RR$ on $\tilde Z$ which is concretely given in the proof of \cite[Theorem~6.9]{Cl77} using local coordinates around the central fiber of $\pi$, and from \cite[(6.8) in p.~244]{Cl77} where the multiplicities of the components of the central fiber are included. The second property then also holds for $\Phi$, since $\mu$ only contracts some irreducible components of $\tilde{F}_0$.
\end{proof}

\subsection{Proof of Theorem~\ref{t:t1}}
In the case the base is a smooth curve, we give the following estimate on the first betti number of the singular fibers, which implies the first part of Theorem~\ref{t:t1}. 
\begin{theorem}
\label{t:t1b1} Let $f : Z \rightarrow C$ be a fibration, where $Z$ is a compact complex $n$-manifold and $C$ is a smooth curve. Then $b^1(Z_y) \leq b^1(F)$
for \textit{all} $y \in C$, where $F$ is the general fiber. Equality holds for some $y \in C$ if and only if the local monodromy $\rho^{(1)}_y$ around $y$ is trivial. 
\end{theorem}
\begin{proof}
The result is local, so we just need to consider a fibration $f : U \rightarrow \Delta$ without any critical points in $\Delta^*$.  By shrinking the disc slightly, we may assume that $\partial U = M$ is a smooth manifold of real dimension $2n-1$ which fibers over the circle, that is, the restriction  $f': M \rightarrow S^1$ is a smooth fiber bundle. By Proposition~\ref{p:def}, we can also assume that $f^{-1}(\Delta)$ deformation retracts onto $f^{-1}(0) \equiv Z_0$. Denote the general fiber by $F$. In the following, unless otherwise indicated, cohomology and homology groups are with real coefficients. By the deformation retraction property and Lefschetz duality on compact manifolds with boundary (see \cite[Corollary~VI.9.3]{Bredon}), we have 
\begin{align}
H^1(Z_0) \simeq H^1(U) \simeq H_{2n-1}(U, \partial U). 
\end{align}
Since $Z_0$ is of real dimension $2n-2$, $H_{2n-1}(U) = 0$, so 
the long exact sequence in relative homology is 
\begin{equation}
\begin{tikzcd} 
 0 \arrow[r] & H_{2n-1}(U, \partial U) \arrow[r] & H_{2n-2}(M) \arrow[r,"\iota"] & H_{2n-2}(U)
\arrow[r] & H_{2n-2}(U, \partial U).
\end{tikzcd}
\end{equation}
From Proposition~\ref{p:def} above, under the deformation retraction the smooth fiber class $[F] \in H_{2n-2}(M) $ maps to a nontrivial combination $\sum_j m_j [F_j]$, where $F_j$ are the irreducible components of $Z_0$. This means that the mapping $\iota$ is nontrivial.
Consequently, we obtain
\begin{align}
b^1(Z_0) \leq b^{2n-2} (M) -1 = b^1(M) -1
\end{align}
where the last equality holds by Poincar\'e duality, since $M$ is a manifold. 
The Leray sequence for $f' : M \rightarrow S^1$  is 
\begin{equation}
\begin{tikzcd} 
0 \arrow[r] & H^1(S^1) \arrow[r] & H^1(M) \arrow[r] & H^0(S^1, R^1 f'_* \RR) \arrow[r] & 0. 
\end{tikzcd}
\end{equation}
As mentioned in \eqref{hkrho} above, since $R^1 f'_* \RR$ is a locally constant sheaf on $S^1$, we have that 
\begin{align}
H^0(S^1, R^1 f_* \RR)  \simeq H^1_{\rho}(F),
\end{align}
where $H^1_{\rho}(F)$ are the invariants of the monodromy representation on $H^1(F)$.
Consequently, 
\begin{align}
b^1(M) = 1 + \dim H^1_{\rho}(F) \leq 1 + b^1(F).
\end{align}
Letting $b^1_{\rho}(F) \equiv \dim H^1_{\rho}(F)$, combining these estimates, 
we obtain 
\begin{align}
b^1(Z_0) \leq b^1_{\rho}(F) \leq b^1(F),
\end{align}
with equality if and only if $H^1_{\rho}(F) = H^1(F)$ which is equivalent to the local monodromy acting trivially. 
\end{proof}

We next need the following lemma, which will be used below in the proofs of Theorem~\ref{t:t1n2} and Theorem~\ref{t:t2n2}. 
\begin{lemma}
\label{l:Euler}
 Let $D \subset X$ be a closed subvariety of the compact complex space $X$. If $Sing(X)$ is discrete, then 
\begin{align}
\label{chiX}
\chi(X) = \chi(X \setminus D) + \chi(D).
\end{align}
\end{lemma}
\begin{proof}
Note that by the Lojasiewicz triangulation theorem, there is a neighborhood $U$ of $D$ in $X$ which strongly deformation retracts onto $D$; see \cite{Loja}. This will be used several times below. 

Case (1): assume that $X$ is smooth. In this case, Lefschetz duality (see \cite[Theorem~6.2.19]{Spanier}) implies that 
\begin{align}
H_j(X \setminus D) \simeq \check{H}^{2\dim_\CC(X)-j} (X,D) \simeq  H^{2\dim_\CC(X)-j} (X,D)
\end{align}
for  $j = 0, \dots, 2 \dim_{\CC}(X)$ where $\check{H}^*(X,D)$ is the direct limit of $H^*(X,U)$ as $U$ ranges over all open subsets of $D$ in $X$, which is equal to the ordinary cohomology $H^*(X,D)$ (as follows from the Lojasiewicz triangulation theorem).  The long exact sequence in cohomology for the pair $(X,D)$ then implies \eqref{chiX}.  

Now we assume that $X$ is not smooth. Let $\pi : \tilde{X} \rightarrow X$ be a resolution; see~\cite{Hiro71}. 
Let $\tilde{D} = \pi^{-1}(D)$ be the total transform of $D$. 

Case (2): assume that $\Sing(X) \subset D$. Then $\tilde{X} \setminus \tilde{D}$ is biholomorphic to $X \setminus D$. Since $\tilde{X}$ is smooth, by Case (1) we have 
\begin{align}
H_j(X \setminus D) \simeq H_j(\tilde{X} \setminus \tilde{D})
\simeq H^{2\dim_\CC(X)-j} (\tilde{X},\tilde{D}).
\end{align}
Next, since $\tilde{D}$ is a neighborhood retract in $\tilde{X}$,
the pair $(\tilde{X}, \tilde{D})$ is a good pair, so from \cite[Proposition~2.22]{Hatcher}, we have the isomorphism
\begin{align}
H^q(\tilde{X}, \tilde{D}) \simeq \tilde{H}^q(\tilde{X} / \tilde{D})
\end{align}
for all $q \geq 0$, where $\tilde{H}^q$ denotes the reduced cohomology, 
and the quotient means $\tilde{D}$ is identified to a point. 
But since $\Sing(X) \subset D$, we also have that $\tilde{X} / \tilde{D}$ is homeomorphic to $X/D$,
so then using Hatcher's Theorem again with the good pair $(X,D)$, we have 
\begin{align}
H_j(X \setminus D) \simeq \tilde{H}^{2\dim_\CC(X)-j} (\tilde{X} / \tilde{D})
\simeq \tilde{H}^{2\dim_\CC(X)-j} (X/D) \simeq H^{2\dim_\CC(X)-j}(X,D),
\end{align}
and then \eqref{chiX} follows as before. 

Finally, we have Case (3): $\Sing(X)$ is not completely contained in $D$. Let $\{p_1, \dots, p_k\}$ denote those singular points not in $D$. We let $\pi' : X' \rightarrow X$ be the resolution at all the singular points $p_j$, and let $E_1, \dots, E_k$ denote the corresponding exceptional divisors. 
Since $D$ does not hit those points $p_j$, $D' = \pi'^{-1}(D)$ is homeomorphic to $D$.  On $X'$, we are in Case~(2) above ($\Sing X' \subset D'$), so we have 
\begin{align}
\chi( X' \setminus D') = \chi(X') - \chi(D') = \chi(X') - \chi(D). 
\end{align}
Since the singularities $p_j$ are isolated, by the conical structure theorem from Milnor \cite[Theorem~2.10]{Milnor_sing} (see also \cite{BV} for the analytic case), there is a neighborhood $U_j$ of $p_j$ which is homeomorphic to a cone on the link $L_j$, which is a smooth manifold. Then a straightforward Mayer-Vietoris argument shows that we have 
\begin{align}
\begin{split}
\chi(X' \setminus D') &= \chi(X \setminus D) + \sum_{j=1}^k ( \chi(E_j) -1) \\
&= \chi(X') - \chi(D) = \chi(X) + \sum_{j=1}^k (\chi(E_j) -1 ) - \chi(D),
\end{split}
\end{align}
which again implies \eqref{chiX}.
\end{proof}
\begin{remark} The above result holds in the algebraic setting with no restriction on the singularities of $X$; see \cite[Section~4.5]{Fultontoric}. With a little more work, it can also be proved in the analytic setting with no restriction on the singularities of $X$. But for the applications in this paper, we only require this additivity if $X$ has isolated singularities. 
\end{remark}
We can now prove the second part of Theorem~\ref{t:t1}. 
\begin{theorem} 
\label{t:t1n2} 
Let $f : Z \rightarrow C$ be a fibration, where $Z$ is a compact complex surface and $C$ is a smooth curve. Let $F$ denote the general fiber. 
Then $\chi(Z) \geq \chi(F) \chi(C)$
with equality if and only if the global monodromy $\rho^{(1)}$ is trivial and every fiber of $f$ is irreducible.
\end{theorem}
\begin{proof}
If $y \in C \setminus \ms{D}$, then $\chi(Z_y) = \chi(F)$. So then by the Leray spectral sequence for $f'$, we conclude that $\chi(Z \setminus f^{-1}(\ms{D}) ) = \chi(F) \cdot \chi( C \setminus \ms{D})$; see\cite[Section~9.3]{Spanier}. Writing $\ms{D} = \{y_1, \cdots, y_k\}$, and $Z_j \equiv f^{-1}(y_j), j = 1, \dots, k$,  by Lemma~\ref{l:Euler} we obtain
 \begin{align}
 \chi(Z) - \sum_{j=1}^k \chi(Z_j)  = \chi(F) \cdot ( \chi(C) - \sum_{j=1}^k 1), 
\end{align}
which rearranges to 
 \begin{align}
 \chi(Z) = \chi(F) \cdot \chi(C) + \sum_{j=1}^k ( \chi(Z_j) - \chi(F)).
\end{align}
So if we show that $\chi(Z_j) \geq \chi(F)$ for all $j = 1, \dots, k$, then we are done with the inequality. Since all fibers are connected curves, $\chi(Z_j) = 1 - b_1(Z_j) + b_2(Z_j)$. 
Since $b_2(Z_j)$ is the number of irreducible components of $Z_j$, we have that $b_2(Z_j) \geq b_2(F) =1$.  So the estimate is reduced to showing $b^1(Z_j) \leq b^1(F)$, which was proved in Theorem \ref{t:t1b1}. The equality $\chi(Z) = \chi(F) \chi(Y)$ holds if and only if $b^1(Z_j) = b^1(F)$ and $b^2(Z_j) = b^2(F)$ for all $j = 1, \dots, k$, which holds if and only if the local monodromy is trivial around every $y_j$ (which is equivalent the global monodromy being trivial) and that all fibers of $f$ are irreducible.  
\end{proof}

\begin{remark}
As mentioned in the introduction, the case of $n=2$ was previously proved in \cite[Chapter~III.11]{BHPV}. Our proof might seem simpler, however, they also proved the much stronger result that 
\begin{align}
\label{ebetter}
\dim H^1(\mathcal{O}_{Z_y}) \leq \dim  H^1(\mathcal{O}_{F}) = g
\end{align} 
for every $y \in C$ (which implies that $b^1(Z_y) \leq 2g$ by the exponential sequence). 
\end{remark}

\subsection{Local invariant cycles theorem}
\label{ss:ict}
We saw in the above proof that
\begin{align}
b^1(F_0) \leq b^1_{\rho}(F) \equiv \dim(H^1_{\rho}(F)).
\end{align} 
In the case of surfaces, these are equal, as can be seen by the following. 
The long exact sequence in homology yields
\begin{equation}
\begin{tikzcd}
 0 \arrow[r] & \RR^{b^1(F_0)} \arrow[r] & \RR^{1 + b^1_{\rho}(F)} \arrow[r,"\iota"] & 
 Ker (  H_{2}(U) \rightarrow H^2(U) ) \arrow[r] & 0.
\end{tikzcd}
\end{equation}
The mapping $H_{2}(U) \rightarrow H^2(U)$ is given by Poincar\'e duality, and it is not hard to see that this agrees with the intersection form on $H_2(U)$; see \cite[Proposition~2.3.4]{Dimca}. By Zariski's Lemma \cite[Lemma~III.8.2]{BHPV}, the intersection form on $H_2(U)$ has exactly a $1$-dimensional kernel given by the class of a smooth fiber, so the local invariant cycles theorem follows.

In higher dimensions, the above gives
\begin{equation}
\begin{tikzcd}
 0 \arrow[r] & \RR^{b^1(F_0)} \arrow[r] & \RR^{1 + b^1_{\rho}(F)} \arrow[r,"\iota"] & 
 Ker (  H_{2n-2}(U) \rightarrow H^2(U) ) \arrow[r] & 0.
\end{tikzcd}
\end{equation}
So one could say more in higher dimensions if something is known about the Poincar\'e duality mapping  $H_{2n-2}(U) \rightarrow H^2(U)$. To this end, we just state here the following obvious case. 
\begin{corollary}
\label{c:ict}
Let $f: Z \rightarrow C$ be as above, where $\dim(Z) = n$. If a singular fiber $F_0$ is irreducible, then the local invariant cycles theorem holds on $H^1$, that is, $b^1(F_0) = b^1_{\rho}(F)$.
\end{corollary}
\begin{proof} If $F_0$ is irreducible, then $H_{2n-2}(U) \simeq H_{2n-2}(F_0) \simeq \RR$. So the above exact sequence must be 
\begin{equation}
\begin{tikzcd}
 0 \arrow[r] & \RR^{b^1(F_0)} \arrow[r] & \RR^{1 + b^1_{\rho}(F)} \arrow[r,"\iota"] & 
 \RR \arrow[r] & 0,
\end{tikzcd}
\end{equation}
from which we conclude that $b^1(F_0) =  b^1_{\rho}(F)$.
\end{proof}
\begin{remark} As mentioned in the Introduction, if $Z$ were assumed to be K\"ahler, then the local invariant cycles theorem is true in higher dimensions; see \cite{Cl77}. There is a well-known example of Kodaira which is a family of smooth Hopf surfaces over the disc with central fiber $F_0$ a Hirzebruch surface with the zero and infinity sections identified; see \cite{Kodaira1964}.  As Clemens points out, the total space is not K\"ahler and the local invariant cycles theorem fails on $H^3$. However, the local invariant cycles theorem still holds on $H^1$, so this gives a non-K\"ahler example of Corollary~\ref{c:ict}.  
\end{remark}

\section{Surface singularities}
\label{s:sing}
In this section, we will discuss the class of \textit{rational tree} singularities, which will be used in Section \ref{s:tfc}. We first recall the following notion concerning resolutions of normal surface singularities; see for example \cite[p.~134]{Dur79} or \cite[p.~50]{Dimca}. 
\begin{definition}
\label{d:gsing}
Let $Y$ be a normal complex surface and $y\in Y$ a singularity of $Y$. A resolution $\mu:\tilde Y\rightarrow Y$ of $y$ is called {\em very good} if the exceptional divisor $E=\mu\inv(y)$
satisfies the following three properties:
(i)   any irreducible component of $E$ is non-singular, (ii)  any intersection of two distinct components of $E$ is transverse, and (iii) any two distinct components are either disjoint or intersect in exactly one point.
\end{definition}
For background on resolutions of normal surface singularities, we refer the reader to \cite{GH, Ishii, LaBook}. Any normal surface singularity is isolated.
There exists a unique resolution $\mu_1:\tilde Y_1\rightarrow Y$ which is minimal in the sense that any resolution $\mu:\tilde Y\rightarrow Y$ factors through $\mu_1$; namely, there exists a unique holomorphic mapping $\aaa:\tilde Y\rightarrow \tilde Y_1$ such that $\mu = \mu_1\circ\aaa$. This minimality is equivalent to the condition that no $(-1)$-curve is contained in the exceptional curve of $\mu_1$.
The minimal resolution is not necessarily a very good resolution, but we can always obtain a very good resolution from the minimal resolution after finitely many blow-ups;
see for example \cite[p.~496]{Brieskorn1986}, \cite[Chapter~8B]{Mumford1976}, or \cite[II.7]{BHPV}.
\begin{example}
\label{ex:cusp}
Let $y$ be a singularity that has a resolution whose exceptional curve consists of a single curve which is a rational curve with one node. This resolution is not very good because the unique component of the exceptional curve is singular. Blowing up at the node of the rational curve gives a resolution which is still not very good because there are two intersection points. Blowing up one at one of these intersection points finally yields a very good resolution. 
\end{example}
There exists a unique minimal very good resolution which is minimal amongst the collection of all very good resolutions. As the above example shows, the minimal very good resolution can contain a $(-1)$-curve in the exceptional divisor. To any very good resolution, we can associate the {\textit{dual graph}} of the resolution, denoted by $\Gamma$, which has a vertex for each irreducible component of $E$, with an edge between two vertices if they intersect. We next introduce the following kinds of surface singularities which appear in Theorem~\ref{t:t2}. 
\begin{definition}\label{d:tsing}
Let $Y$ be a normal complex surface and $y\in Y$ a singularity of $Y$. Let $\mu:\tilde Y \rightarrow Y$ be a very good resolution of $y$ and $E :=\mu\inv(y)$ the exceptional divisor of $\mu$. We say that $y$ is a {\em rational tree singularity} of $Y$ if all irreducible components of $E$ are rational, and the dual graph of $E$ is a tree. 
\end{definition}
Since there exists a minimal very good resolution, it follows that if the above condition holds for \textit{some} very good resolution, then the same holds for \textit{all} very good resolutions. 

Artin defined a \textit{rational singularity} in \cite{Artin1966}. These are necessarily rational tree singularities (see \cite[Lemma~1.3]{Brieskorn1967}), but the class of rational tree singularities is strictly larger than the class of rational singularities, as the following example shows.
\begin{example}
\label{ex:IV}
Let $F$ be a singular fiber of type IV of an elliptic fibration in Kodaira's notation. So $F$ consists of three $(-2)$ curves meeting at a point.
Take any one of them and also a point on it that is not a triple point, and let $\tilde F$ be the proper transform of $F$ under the blowing up at the point. Then the intersection matrix of $\tilde F$ is negative definite and hence it can be contracted to a point complex analytically by Grauert's criterion; see \cite[Theorem~III.2.1]{BHPV} and \cite{GrauertUber}. By blowing up the triple point of $\tilde F$, we obtain that 
the contracted point is a rational tree singularity. But it is not a rational singularity because $p_a(\tilde F) = 1$ while the arithmetic genus of the exceptional divisor of a rational singularity always vanishes by Artin's characterization; see \cite[Theorem~III.3.2]{BHPV}.
\end{example}

The next proposition gives an alternative characterization of the singularities in Definition~\ref{d:tsing}, in terms of the first Betti number of the exceptional curve of a very good resolution.
\begin{proposition}
\label{p:scri}
Under the assumptions in Definition \ref{d:tsing}, $y$ is a rational tree singularity if and only if $b_1(E) = 0$.
\end{proposition}
\begin{proof}
Let $\CCC$ be the dual graph of $E$. If $b$ is the number of closed cycles in $\CCC$ and $g$ is the sum of the genera of the irreducible components of $E$, then by \cite[Proposition~2.3.1]{Dimca}, 
\begin{align}\label{b1E}
b_1(E) = 2g + b.
\end{align}
From the definition, $y$ is a rational tree singularity if and only if $g=b=0$, which is equivalent to $b_1(E)=0$ as $b\ge 0$ and $g\ge 0$.
\end{proof}

We recall the following basic result which is proved in \cite[Chapter~1]{Dimca}. 
\begin{proposition}
\label{p:cone}
For an isolated surface singularity $p \in Y$, there exists a contractible neighborhood $\Delta_p$ of $p$ which is homeomorphic to a cone over a compact $3$-manifold $L$.
\end{proposition}
The $3$-manifold $L$ is called the \textit{link} of the singularity. 
We will also need the following characterization of rational tree singularities.
\begin{proposition}
\label{p:scri2}
Let $Y$ be a normal complex surface and $y\in Y$ a singularity of $Y$.
Then $y$ is a rational tree singularity if and only if $b_1(L) = 0$, that is, $L$ is a $\QQ$-homology $3$-sphere. 
\end{proposition}
\begin{proof}
Let $\mu : \tilde{Y} \rightarrow Y$ be a very good resolution with exceptional divisor $E = \mu^{-1}(y)$. Then by \cite[Proposition~2.3.4]{Dimca}, $b_1(L) = b_1(E) = 0$, so this follows immediately from Proposition~\ref{p:scri}. 
\end{proof}

\section{Conic bundle over a surface}
\label{s:tfc}
A \textit{conic bundle} is another name for a fibration $f: Z \rightarrow Y$ with generic fiber a rational curve. The following example is useful to keep in mind. 
\begin{example} Let $Z \subset \PP_z^2 \times \PP_w^2$ be defined by 
\begin{align}
Z = \{ z_0 w_0^2 + z_1 w_1^2 + z_2 w_2^2 = 0 \ | \ [ z_0, z_1, z_2] \in \PP^2,
 [ w_0, w_1, w_2] \in \PP^2 \},
\end{align}
with $f : Z \rightarrow Y = \PP^2_z$ projection onto the first factor. It is easy to verify that $Z$ is a smooth compact threefold. The generic fiber of $f$ is a smooth conic. The discriminant locus $\ms D$ consists of the $3$ coordinate lines in $\PP^2_z$. There are $3$ non-ordinary points which are the intersection points. Over the ordinary points, the fibers are $2$ lines intersecting at a point, but over the non-ordinary points, the fiber is a double line. 
\end{example}
We will next prove the following result in the case of a threefold with conic bundle structure over a surface with rational tree singularities, which is the first part of Theorem~\ref{t:t2}.
\begin{theorem}
\label{t:tf1}
 Let $f : Z \rightarrow Y$ be a fibration, where $Z$ is a compact connected threefold and $Y$ is a compact surface with rational tree singularities. Assume that the generic fiber $F$ of $f$ is a smooth rational curve.
Then $b^1(Z_y) = 0$ for all $y \in Y$. 
\end{theorem}

\begin{proof}
By Proposition~\ref{p:non}, the set of non-ordinary points on the discriminant locus is
a finite set, so we write $N = \{p_1, \dots, p_k\}$.
Using Corollary~\ref{c:5term}, we consider the commutative diagram
\begin{equation}
\begin{tikzcd} 
 H^1(Y) \arrow[r] \arrow[d,"\alpha"] & H^1(Z) \arrow[r] \arrow[d,"\beta"] & H^0(Y, R^1f_* \RR) \arrow[r]
\arrow[d,"\gamma"] & H^2(Y) \arrow[d,"\delta"]  \\
 H^1(Y') \arrow[r]  & H^1(Z') \arrow[r]  & H^0(Y', R^1f'_* \RR) \arrow[r] & H^2(Y') 
\end{tikzcd}
\end{equation}
where $Y' = Y \setminus N$ and $Z' = Z \setminus f^{-1}(N)$. 

By excision, for any $i \geq 0$, we have that 
\begin{align}
H^i(Y, Y') \simeq \bigoplus_{j = 1}^k H^i( \Delta_j, \Delta_j \setminus \{p_j\}),
\end{align}
where $\Delta_j$ is a neighborhood of $p_j$ as in Proposition~\ref{p:cone}. 
Since $\Delta_j$ is contractible, the long exact sequence in cohomology 
for the pair $ (\Delta_j, \Delta_j \setminus \{p_j\})$ yields
\begin{align}
H^i(  \Delta_j, \Delta_j \setminus \{p_j\}) \simeq \tilde{H}^{i-1} ( \Delta_j \setminus \{p_j\}) 
\simeq \tilde{H}^{i-1}(L_j), 
\end{align}
for $i \geq 1$, where $L_j$ is the link at $p_j$, and $\tilde{H}^*$ denotes reduced cohomology. By Proposition~\ref{p:scri2}, $L_j$ is a rational homology sphere, so we conclude that $H^i(Y, Y')$ vanishes for $i = 1,2,3$, so $\alpha$ and $\delta$ are isomorphisms. 

By Poincar\'e-Alexander-Lefschetz duality (see \cite[Theorem~VI.8.3]{Bredon}), we have 
\begin{align}
\label{e:PAD}
H_1(Z,  Z \setminus f^{-1}(N))\simeq \check{H}^5( f^{-1}(N)) \simeq H^5(f^{-1}(N)),
\end{align}
where $\check{H}^1( f^{-1}(N))$ is the direct limit of $H^1(U)$ as $U$ ranges over all open subsets of $f^{-1}(N)$ in $Z$, which is equal to the ordinary cohomology of $f^{-1}(N)$ (as follows from the Lojasiewicz triangulation theorem). 
Since all fibers are at most of dimension $2$, the right hand side of \eqref{e:PAD} vanishes, and consequently $\beta$ is injective. By the four lemma (see \cite[p.~14]{MacLane}), we have an injection
\begin{equation}
\begin{tikzcd}
\label{h00}
0 \arrow[r] & H^0(Y, R^1f_* \RR) \arrow[r,"\gamma"] & H^0(Y', R^1f'_* \RR).
\end{tikzcd}
\end{equation}
Recall from Proposition~\ref{p:stalk} that for any $y$ we have an isomorphism $(R^1 f'_* \RR)_y \simeq H^1(Z_y;\RR)$. Away from the discriminant locus, this vanishes by our assumption that the general fiber is a rational curve.  At an ordinary point of the discriminant locus, there is a transverse curve $C$ which lifts to a smooth surface, and we can use the argument in the proof of 
Theorem~\ref{t:t1b1} to see that $H^1(Z_y;\RR) = 0$ at ordinary points. We conclude that $R^1 f'_* \RR = 0$.
So then $R^1f_*\RR$ is a skyscraper sheaf supported at the non-ordinary points. But \eqref{h00} implies that $H^0(Y, R^1 f_* \RR)  = 0$, which means the stalk must also be trivial at the non-ordinary points. Consequently, $R^1f_* \RR = 0$ is the trivial sheaf, so $H^1(Z_y;\RR) = 0$ for all $y \in Y$. 
\end{proof}

The following result gives the Euler characteristic estimate in Theorem~\ref{t:t2}.
\begin{theorem}
\label{t:t2n2}
If in addition to the assumption in Theorem~\ref{t:tf1}, the fibers of $f$ are equidimensional, then
\begin{align}
\label{t2chi1}
\chi(Z) \geq 2 \cdot \chi(Y),
\end{align}
with equality if and only if every fiber is homeomorphic to $\PP^1$. 
\end{theorem}
\begin{proof}
If all fibers are equidimensional, Theorem~\ref{t:tf1} implies that $\chi(Z_y) \geq 2$ for every $y \in Y$. 
We next use Proposition \ref{p:topo} to conclude that there is a finite 
subset $\mathscr{N} \subset \mathscr{D}$, so that $f : f^{-1} ( \ms{D} \setminus \ms{N}) \rightarrow \ms{D} \setminus \ms{N}$ is a topological fiber bundle over each connected component, with the fiber having Euler characteristic at least $2$. So from \cite[Section~9.3]{Spanier}, we have 
\begin{align}
\chi(f^{-1} (\ms{D} \setminus \ms{N})) \geq  2 \chi( \ms{D} \setminus \ms{N}).
\end{align}
To calculate the Euler characteristic, we use Lemma~\ref{l:Euler} and compute
\begin{align}
\begin{split}
\chi(Z) &= \chi ( f^{-1}(Y \setminus \ms{D})) + \chi( f^{-1}(\ms{D} \setminus \ms{N})) + \chi(f^{-1}(\ms{N}))\\
& \geq 2 \chi(Y \setminus \ms{D}) + 2 \chi( \ms{D} \setminus \ms{N}) + 2 \chi(\ms{N}) = 2 \chi(Y).
\end{split}
\end{align}
Finally, if we have equality in \eqref{t2chi1}, then we conclude that any singular fiber $Z_y$ is irreducible with $b^1(Z_y) = 0$. Let $\nu: \check{Z}_y \rightarrow Z_y$ denote the normalization. Since $Z_y$ is irreducible, so is $\check{Z}_y$. The normalization mapping must be injective since $b^1(Z_y) = 0$. Consequently, $\check{Z_y} \simeq \PP^1$, and $Z_y$ is homeomorphic to $\PP^1$. 
\end{proof}

\section{Leray spectral sequence}
\label{s:Leray}
As mentioned in the Introduction, in this section we prove a naturality property of the Leray spectral sequence with respect to inclusions of open subsets.
All sheaves considered will be sheaves of abelian groups.
Let $M$ be a topological space and $\ms S\rightarrow M$ a sheaf over $M$. 
As in \cite[p.~37]{GunningIII}, let $\ms A^0(\ms S)$ be the sheaf over $M$ that is canonically obtained from $\ms S$ by considering discontinuous sections, and let $\ms B^1(\ms S)$ be the quotient sheaf $\ms A^0(\ms S)/\ms S$. 
Repeating this procedure, by putting $\ms A^p(\ms S)=\ms A^0(\ms B^p(\ms S))$ and $\ms B^{p+1}(\ms S) = \ms A^{p}(\ms S)/\ms B^{p}(\ms S)$ for any $p\ge 0$, we obtain the exact sequence
\begin{align}\label{can1}
0\lras \ms S \lras \ms A^0(\ms S) \stackrel{\ptl}\lras \ms A^1(\ms S)\stackrel{\ptl}\lras \ms A^2(\ms S)\lras\cdots,
\end{align}
which is the canonical resolution of $\ms S$. 
All sheaves $\ms A^p(\ms S)$ are flabby, and the cohomology groups of the associated sequence of global sections are the cohomology groups $H^p(M, \ms S)$.

For any non-empty 
subset $M'\subset M$, we use the notation $\ms S'$ to mean the restriction of a sheaf $\ms S$ onto $M'$.
The following proposition can be easily seen just by faithfully following the definition of the canonical resolution and noting that any point of $M'$ has a neighborhood in $M'$ that is open in $M$, provided that $M'$ is an open subset of $M$.

\begin{proposition}\label{p:001} If $M'$ is an open subset of $M$, then
for any $p\ge 0$, there exist natural isomorphisms
$$\big(\ms A^p(\ms S)\big)' \simeq \ms A^p(\ms S')\qandq
\big(\ms B^p(\ms S)\big)' \simeq \ms B^p(\ms S')$$
of sheaves over $M'$. Under the first isomorphism, the sequence obtained from \eqref{can1} as the restriction onto $M'$ is exactly the canonical resolution of the sheaf $\ms S'$ over $M'$.
\end{proposition}

Next, let $\ms S$ be a sheaf over a topological space $M$ and $f:M\rightarrow N$ a continuous mapping to another topological space $N$. Further, let $N'\subset N$ be any non-empty open subset and suppose that the open subset $M':=f\inv(N')$ is non-empty.
We denote $f':=f|_{M'}:M'\rightarrow N'$. For a sheaf $\ms T$ over $N$, the restriction onto $N'$ will be also denoted by $\ms T'$.
\begin{proposition}\label{p:002}
For any $p\ge 0$, there is a natural isomorphism 
\begin{align}\label{isom01}
(R^pf_*\ms S)'\simeq R^pf'_*\ms S'.
\end{align}
For a short exact sequence 
\begin{align}\label{ases1}
0\lras \ms S_1\lras \ms S_2\lras \ms S_3\lras 0
\end{align}
of sheaves over $M$, under the isomorphism \eqref{isom01} 
for each sheaf below,
the restriction of the direct image exact sequence
\begin{align}\label{ales1}
0 \lras f_*\ms S_1\lras f_*\ms S_2\lras f_*\ms S_3\lras
R^1f_*\ms S_1\lras\cdots
\end{align}
onto $N'$ is identified with the direct image exact sequence
$$
0 \lras f'_*\ms S'_1\lras f'_*\ms S'_2\lras f'_*\ms S'_3\lras
R^1f'_*\ms S'_1\lras\cdots
$$
over $N'$ obtained from the restriction exact sequence $0\rightarrow \ms S'_1\rightarrow \ms S'_2\rightarrow \ms S'_3\rightarrow 0$ of \eqref{ases1} onto $M'$ as the direct image sequence under $f'$.
\end{proposition}
This proposition can also be shown by faithfully following the definition of the direct image sheaf including the higher ones
and the proof for obtaining \eqref{ales1} from \eqref{ases1}
(see \cite[Theorem~F.4]{GunningIII}), and we omit a proof.

In the setting of the previous proposition, for any $p,q\ge 0$, we put
\begin{align}\label{Apq}
A^{p,q}&:= \CCC\big(N, \ms A^p\big(f_*\,\ms A^q(\ms S)\big)\big),\\
B^{p,q}&:= \CCC\big(N', \ms A^p\big(f'_*\,\ms A^q(\ms S')\big)\big).
\label{Bpq}
\end{align}
where $\CCC$ denotes the space of (continuous) sections.
For any $p,q\ge 0$, we denote
$
\ptl: A^{p,q} \rightarrow A^{p+1,q}
$
for the homomorphism induced from the sheaf homomorphism 
$\ptl:\ms A^p\big(f_*\,\ms A^q(\ms S)\big)
\rightarrow 
\ms A^{p+1}\big(f_*\,\ms A^q(\ms S)\big)
$
in the canonical resolution of the sheaf $f_*\,\ms A^q(\ms S)$ on $N$.
Similarly, we denote
$
\ptl': B^{p,q} \rightarrow B^{p+1,q}
$
for the homomorphism induced from the analogous sheaf homomorphism on $N'$.
The homomorphisms $(-1)^{p+q}\ptl:A^{p,q}\rightarrow A^{p+1,q}$ and 
$(-1)^{p+q}\ptl':B^{p,q}\rightarrow B^{p+1,q}$ will be used as the first differential for the double complexes formed by $A^{p,q}$ and $B^{p,q}$ respectively.

To discuss another differential for these double complexes,
we recall that
for any $p\ge 0$, any sheaf homomorphism $\theta:\ms R_1\rightarrow \ms R_2$ over $N$ naturally induces a sheaf homomorphism $\theta^{(p)}:\ms A^p(\ms R_1)\rightarrow \ms A^p(\ms R_2)$; see the proof of \cite[Theorem~D.2]{GunningIII}.
Applying this to the homomorphism $f_*\big(\ms A^q(\ms S)\big)\rightarrow f_*\big(\ms A^{q+1}(\ms S)\big)$ induced from $\ptl:\ms A^q(\ms S)\rightarrow \ms A^{q+1}(\ms S)$ and $f$,
for any $p,q\ge 0$, we obtain a homomorphism
\begin{align}\label{hom100}
\ms A^p\big(f_*\big(\ms A^q(\ms S)\big)\big)\lras \ms A^p\big(f_*\big(\ms A^{q+1}(\ms S)\big)\big).
\end{align}
Taking the induced mapping on the spaces of global sections, we obtain the second differential
$
\ol\ptl: A^{p,q} \rightarrow A^{p,q+1}.
$
Similarly, we obtain the second differential for $B^{p,q}$, which we  denote by
$
\ol\ptl': B^{p,q} \rightarrow B^{p,q+1}.
$
We then have
\begin{proposition}\label{p:003}
In the setting of the previous proposition, for any $p,q\ge 0$, 
if $\phi:A^{p,q} \rightarrow B^{p,q}$ is the restriction homomorphism of the domain,
then it satisfies 
the commutativity $\phi\circ\ptl = \ptl'\circ\phi$ and 
$\phi\circ\ol\ptl = \ol\ptl'\circ\phi$.
\end{proposition}
\begin{proof}
From the previous two propositions, we have natural isomorphisms
\begin{align}\label{isom007}
\Big(
\ms A^p\big(f_*\ms A^q(\ms S)\big)
\Big)'
&\simeq
\ms A^p\Big(\big(f_*\ms A^q(\ms S)\big)'\Big)
\quad ({\text{by Proposition\,\ref{p:001}}})\\
&\simeq
\ms A^p\Big(f'_*\big(\big(\ms A^q(\ms S)\big)'\big)\Big)
\quad ({\text{by Proposition\,\ref{p:002}}})\label{isom008}\\
&\simeq
\ms A^p\Big(f'_*\big(\ms A^q(\ms S')\big)\Big)
\quad ({\text{by Proposition\,\ref{p:001}}})\label{isom009}
\end{align}
Therefore, we obtain the commutative diagram
\begin{equation}
\begin{CD}\label{}
\ms A^p\big(f_*\ms A^q(\ms S)\big) @>{\ptl}>> 
\ms A^{p+1}\big(f_*\ms A^q(\ms S)\big)\\
@AAA @AAA\\
\ms A^p\big(f'_*\ms A^q(\ms S')\big) @>{\ptl'}>> 
\ms A^{p+1}\big(f'_*\ms A^q(\ms S')\big)\\
\end{CD}
\end{equation}
where the two vertical maps are the inclusion mappings as open subsets respectively.
Taking the spaces of sections over $N$ and $N'$ of these sheaves respectively and letting ${\phi}$ be the restriction of the domains (from $N$ to $N'$), we obtain the first commutativity.

Similarly, using the homomorphism \eqref{hom100}, we have the following commutative diagram
\begin{equation}
\begin{CD}\label{}
\ms A^p\big(f_*\big(\ms A^q(\ms S)\big)\big) @>>> 
\ms A^{p}\big(f_*\big(\ms A^{q+1}(\ms S)\big)\big)\\
@AAA @AAA\\
\Big(\ms A^p\big(f_*\big(\ms A^q(\ms S)\big)\big)\Big)' @>>> 
\Big(\ms A^{p}\big(f_*\big(\ms A^{q+1}(\ms S)\big)\big)\Big)'
\end{CD}
\end{equation}
where the vertical arrows are the inclusion mappings as open subsets.
But from \eqref{isom007}--\eqref{isom009}, 
the lower left (resp.\,the lower right) is canonically isomorphic to 
$\ms A^p\big(f'_*\big(\ms A^q(\ms S')\big)\big)$
(resp.
$\ms A^{p}\big(f'_*\big(\ms A^{q+1}(\ms S')\big)\big)$).
Taking the spaces of global sections, 
the restriction mapping satisfies
the second commutativity.
\end{proof}
In the above situation, the homomorphisms $\ptl$ and $\ol\ptl$ satisfy $\ptl^2 = \ol\ptl^2 = \ptl\ol\ptl - \ol\ptl\ptl=0$, which means that 
the collection $A:=\{A^{p,q},(-1)^{p+q}\ptl,\ol\ptl\}_{p,q\ge0}$
constitutes a double cochain complex.
For any $n\ge 0$, put $A^n = \oplus_{p+q=n}A^{p,q}$ and let $\{A^n, (-1)^{p+q}\ptl + \ol \ptl\}_{n\ge 0}$ be the single cochain complex  associated with the double complex $A$.
Let $_1F^i(A^n) := \oplus_{p+q = n, p\ge i}\, A^{p,q}$
and $_2F^i(A^n) := \oplus_{p+q = n, q\ge i}\, A^{p,q}$ be the two filtrations on the single complex.
Each of these gives rise to a spectral sequence \cite[Theorem~C.12]{GunningIII}, for which we denote by $\{{}_{1}E_r^{p,q}(A), \,_1d_r\}$ and  $\{{}_{2}E_r^{p,q}(A), \,_2d_r\}$ respectively. The second terms of these are, respectively, 
\begin{align}\label{E2}
_{1}E_2^{p,q}(A) = H_{\ptl}^p\big(H_{\ol\ptl}^q(A^{*,*})\big)
\qandq
_{2}E_2^{p,q}(A) = H_{\ol\ptl}^p\big(H_{\ptl}^q(A^{*,*})\big).
\end{align}
Both of these spectral sequences are convergent, and therefore the limit terms are components of the graded modules associated with the two filtrations induced on the cohomology group $H^n(A^*)$ of the single complex. We denote the analogous two spectral sequences obtained from the double complex $B:=\{B^{p,q},(-1)^{p+q}\ptl',\ol\ptl'\}_{p,q\ge0}$ by $\{{}_{1}E_r^{p,q}(B), \,_1d'_r\}$ and  $\{{}_{2}E_r^{p,q}(B), \,_2d'_r\}$ respectively.

Proposition \ref{p:003} means that the collection $\{\phi:A^{p,q}\rightarrow B^{p,q}\}_{p,q\ge 0}$ 
formed by the restriction mapping
is a homomorphism from 
$A$ to $B$ as double cochain complexes.
By \cite[Corollary~C.15]{GunningIII}, any homomorphism between double cochain complexes naturally induces 
homomorphisms ${}_{\nu}\phi_r:{}_{\nu}E_r^{p,q}(A)\rightarrow {}_{\nu}E_r^{p,q}(B)$ for $\nu=1,2$ between the terms of the associated spectral sequences for any $p,q,r\ge 0$ and they commute with the differentials $_{\nu}d_r$ and $_{\nu}d'_r$.
Therefore, we have obtained: 
\begin{proposition}
For any $\nu=1,2$ and $p,q,r\ge 0$, there is a commutative diagram 
\begin{equation}
\begin{CD}\label{d:01}
{}_{\nu}E_r^{p,q}(A) @>{{}_{\nu}d_r}>> {}_{\nu}E_r^{p+r,q-r+1}(A)\\
@V{{}_{\nu}\phi_r}VV @VV{{}_{\nu}\phi_r}V\\
{}_{\nu}E_r^{p,q}(B) @>{{}_{\nu}d'_r}>> {}_{\nu}E_r^{p+r,q-r+1}(B)
\end{CD}
\end{equation}
where $_{\nu}E_r^{i,j} = 0$ if $j<0$.
\end{proposition}
Now, as in the proof of \cite[Theorem~F.5]{GunningIII},
$$
H_{\ptl}^q(A^{*,p})
\simeq
\begin{cases}
\CCC(M,\ms A^p(\ms S)) & {\text{if $q=0$,}}\\
0 & {\text{if $q>0$.}}
\end{cases}
$$
and therefore from \eqref{E2} 
and the definition of the sheaf cohomology group,
$$
_{2}E_2^{p,q}(A) = H_{\ol\ptl}^p\big(H_{\ptl}^q(A^{*,*})\big)
\simeq
\begin{cases}
H^p(M,\ms S)& {\text{if $q=0$,}}\\
0& {\text{if $q>0$.}}
\end{cases}
$$
On the other hand, as in the proof of \cite[Theorem~F.5]{GunningIII}
for any $p,q\ge 0$,
\begin{align}\label{isom010}
H^q_{\ol\ptl}(A^{p,*})
\simeq
\CCC\big(N, \ms A^p(R^qf_*\ms S)\big),
\end{align}
and therefore, from \eqref{E2} 
and the definition of the sheaf cohomology group,
\begin{align}\label{isom011}
_{1}E_2^{p,q}(A)=H^p_{\ptl}\big(H^q_{\ol\ptl}(A^{*,*})\big)
\simeq
H^p(N,R^qf_*\ms S).
\end{align}
Using Propositions \ref{p:001}--\ref{p:003}, by faithfully following the proof of \eqref{isom010} and \eqref{isom011}, we can see that 
the mapping $_1\phi_2: {}_1 E_2^{p,q}(A)\rightarrow {}_1 E_2^{p,q}(B)$ in the commutative diagram \eqref{d:01} for the case $r=2$ is just the restriction mapping from $N$ to $N'$. 
This finishes the proof of Theorem~\ref{t:Leray}. Corollary~\ref{c:5term} follows easily from this. 

\begin{remark} With more work, this naturality can also be seen to follow from naturality of the Grothendieck spectral sequence; see~\cite{Haas}. 
\end{remark}

\bibliographystyle{plain}
\bibliography{Survey}

\end{document}